\documentclass[a4paper,draft,reqno,12pt]{amsart}
\usepackage[utf8]{inputenc}
\usepackage[T1,T2A]{fontenc}
\usepackage[english]{babel}
\usepackage{amsmath}
\usepackage{amssymb}
\usepackage{amscd}
\usepackage{amsthm}
\usepackage{euscript}
\usepackage[all]{xy}
\newtheorem{proposition}{Proposition}

\newtheorem{problem}{Problem}

\newtheorem{lemma}{Lemma}
\newtheorem{theorem}{Theorem}

\newtheorem{corollary}{Corollary}
\theoremstyle{definition}

\newtheorem{example}{Example}
\theoremstyle{remark}

\DeclareMathOperator{\SAut}{SAut}
\DeclareMathOperator{\Lie}{Lie}
\DeclareMathOperator{\Aut}{Aut}
\DeclareMathOperator{\SL}{SL}

\DeclareMathOperator{\Sp}{Sp}
\DeclareMathOperator{\red}{red}
\DeclareMathOperator{\ses}{ss}
\DeclareMathOperator{\mult}{sem}
\DeclareMathOperator{\add}{add}

\def\Ga{{\mathbb G}_{\text{a}}}
\def\Gm{{\mathbb G}_{\text{m}}}
\def\Ru{R_{\text{u}}}

\def\KK{{\mathbb K}}

\def\AA{{\mathbb A}}
\def\gg{\mathfrak{g}}
\def\nn{\mathfrak{n}}
\def\sss{\mathfrak{s}}
\def\uu{\mathfrak{u}}
\def\rr{\mathfrak{r}}
\def\sl{\mathfrak{sl}}
\def\sp{\mathfrak{sp}}
\renewcommand{\phi}{\varphi}
\renewcommand{\ge}{\geqslant}

\begin{document}
\date{}
\title{Algebraic groups generated by semisimple elements}
\author{Ivan Arzhantsev}
\address{Faculty of Computer Science, HSE University, Pokrovsky Boulevard 11, Moscow, 109028 Russia}
\email{arjantsev@hse.ru}
\thanks{The article was prepared within the framework of the project ``International Academic Cooperation'' HSE University} 
\maketitle
\begin{abstract}
Given a connected linear algebraic group $G$ over an algebraically closed field of characteristic zero, we describe the subgroup of $G$ generated by all semisimple elements.
\end{abstract} 

\section{Introduction}
\label{sec1}

Let $G$ be a connected linear algebraic group over an algebraically closed field $\KK$ of characteristic zero. We recall some basic facts on the structure of the group $G$. Let $G^{\red}$ and $G^{\ses}$ be a maximal connected reductive subgroup and a maximal connected semisimple subgroup of $G$, respectively. Such subgroups are unique up to conjugation, see~\cite[Theorem~6.5]{OV}. Moreover, $G^{\red} = G^{\ses}\cdot Z$, where $Z$ is the central torus on $G^{\red}$, and $G^{\ses}\cap Z$ is finite. 

Further we use the following notation. Let $G=G_1\ltimes G_2$ be a semidirect product of subgroups $G_1$ and $G_2$, where $G_2$ is normal in~$G$. By $G_1\cdot G_2$ we denote the quotient of the semidirect product $G_1\ltimes G_2$ by a finite normal subgroup. 

Denote by $R(G)$ and $\Ru(G)$ the radical and the unipotent radical of $G$, respectively. By definition, $R(G)$ is the maximal connected normal solvable subgroup of $G$, and $\Ru(G)$ is the maximal normal unipotent subgroup of $G$. We have decompositions
\begin{equation} \label{ttt}
G=G^{\ses}\cdot R(G) \quad \text{and} \quad G=G^{\red}\ltimes\Ru(G);
 \end{equation}
see~\cite[Theorem~6.4]{OV}; the original proof of the second decomposition is given in~\cite{Mo} over any field of characteristic zero. 

Any family $\{X_i, i\in I\}$ of irreducible subvarieties in $G$ passing through the unit generates a closed connected subgroup $H$ in $G$; see~\cite[Theorem~3.1.4]{OV}. If the family is stable under conjugation, then the subgroup $H$ is normal. 

Any unipotent element in $G$ is contained in a one-dimensional unipotent subgroup~\cite[Theorem~3.2.1]{OV} and any semisimple element is contained in a subtorus~\cite[Theorem~22.2]{Hum}. These results and the Jordan decomposition imply that $G$ is generated by one-dimensional subgroups. 

It is well known that any one-dimensional connected linear algebraic group is isomorphic either to the additive group $\Ga=(\KK,+)$ or to the multiplicative group $\Gm=(\KK\setminus\{0\},\times)$ of the ground field~\cite[Theorem~20.5]{Hum}. A \emph{$\Ga$-subgroup} (resp. a \emph{$\Gm$-subgroup}) in $G$  is a closed subgroup isomorphic to $\Ga$ (resp. to $\Gm$). 

Let us say that $G$ is \emph{additively generated} if $G$ is generated by its $\Ga$-subgroups. It is proved in~\cite[Lemma~1.1]{Po} that the following conditions on a connected linear algebraic group $G$ are equivalent:
\begin{enumerate}
\item[(I)]
$G$ is additively generated;
\item[(II)]
$G$ is generated by unipotent elements;
\item[(III)]
$G$ has no nontrivial characters;
\item [(IV)]
$R(G)=\Ru(G)$.
\end{enumerate}
If we denote by $G^{\add}$ the subgroup generated by all $\Ga$-subgroups of a linear algebraic group~$G$, then $G^{\add}$ is a closed normal subgroup in $G$ with 
$G^{\add}=G^{\ses}\ltimes\Ru(G)$. Moreover, we have $G=Z\cdot G^{\add}$, where $Z$ is the central torus in $G^{\red}$. 

In~\cite{Po}, these observations were applied to compute the Makar-Limanov invariant of an affine variety. Also they are useful in the study of so-called flexible varieties. Namely, let $X$ be an affine algebraic variety and $\SAut(X)$ be the subgroup of the automorphism group $\Aut(X)$ generated by all $\Ga$-subgroups. Clearly, if an additively generated linear algebraic group $G$ acts on $X$, then the image of $G$ in $\Aut(X)$ is contained in $\SAut(X)$. In particular, if $G$ acts on $X$ with an open orbit, then $\SAut(X)$ acts on $X$ with an open orbit, and by~\cite{AFKKZ} the action of $\SAut(X)$ on its open orbit is infinitely transitive. 

Let us say that a connected linear algebraic group $G$ is \emph{multiplicatively generated}, if $G$ is generated by all its $\Gm$-subgroups. The aim of this note is to characterize multiplicatively generated groups and to describe the maximal multiplicatively generated subgroup $G^{\mult}$ in a connected linear algebraic group $G$. It is well known that any reductive group is multiplicatively generated, and we are interested in what happens beyond the reductive case for being generated by semisimple elements. In particular, we obtain an effective description of the characteristic subgroup $G^{\mult}$ in a connected linear algebraic group~$G$. The corresponding results are formulated in the next section. 

A natural problem is to describe the subgroup $G^{\mult}$ in a connected linear algebraic group over an arbitrary field, in particular, over an algebraically closed field of positive characteristic. In this note, we make significant use of facts and techniques that are specific to algebraically closed fields and/or fields of characteristic zero. For example, they include  decompositions~(\ref{ttt}), the Lie algebras techniques, the exponential map, and full reducibility of representations of reductive groups. So it is hardly possible to directly transfer the results of this note into a positive characteristic. We expect new ideas and approaches that will allow to deal with the case of an arbitrary ground field. 

\section{Main results}
\label{sec2}
It is easy to see that a connected linear algebraic group $G$ consists of semisimple elements if and only if $G$ is a torus. The next proposition shows that the class of connected linear 
algebraic groups generated by semisimple elements is much wider. 

\begin{proposition} \label{propim}
Let $G$ be a connected linear algebraic group. Then the following conditions are equivalent:
\begin{enumerate}
\item
$G$ is multiplicatively generated;
\item
$G$ is generated by semisimple elements;
\item
any homomorpism $G\to\Ga$ is trivial;
\item 
$G$ has no proper normal subgroup containing $G^{\red}$;
\item 
the derived subgroup $[G,G]$ equals $G^{ss}\ltimes\Ru(G)$.
\end{enumerate}
\end{proposition}

Now we come to a description of the subgroup $G^{\mult}$ generated by all semisimple elements in a connected linear algebraic group $G$. It is convenient to switch to Lie algebras of algebraic groups in question. Denote by $\gg=\Lie(G)$ the tangent algebra of a linear algebraic group~$G$. The nilpotent ideal $\nn=\Lie(\Ru(G))$ of $\gg$ is a $G^{\red}$-module with respect to the adjoint action. Denote by $\nn_1$ the sum of all non-trivial simple $G^{\red}$-submodules in $\nn$. Let $\sss$ be the subalgebra in $\nn$ generated by $\nn_1$ and $U=\exp(\sss)\subseteq\Ru(G)$ be the corresponding unipotent subgroup.  

\begin{theorem} \label{tmain}
Let $G$ be a connected linear algebraic group and $G^{\mult}$ be the subgroup of $G$ generated by all semisimple elements. Then
$$
G^{\mult}=G^{\red}\ltimes U.
$$
\end{theorem} 

Let us recall that a subgroup $G$ of a reductive group $F$ is \emph{regular} if $G$ contains a maximal torus $T$ of $F$. 

\begin{corollary} \label{cor1}
Let $G$ be a regular subgroup of a reductive group $F$. Then $G$ is generated by semisimple elements. In particular, any parabolic subgroup of a reductive group is generated by semisimple elements.
\end{corollary} 

\begin{example}
Consider a solvable linear algebraic group
$$
G=\left\{\left(
 \begin{array}{cccc}
 1 & a & d & b\\
 0 & t & 0 & c\\
 0 & 0 & 1 & e\\
 0 & 0 & 0 & 1 \\
 \end{array}
 \right);  \ t\in\KK\setminus\{0\}, \ a,b,c,d,e \in \KK \right\}.
$$
The subgroup $G^{\mult}$ is
$$
G^{\mult}=\left\{\left(
 \begin{array}{cccc}
 1 & a & 0 & b\\
 0 & t & 0 & c\\
 0 & 0 & 1 & 0\\
 0 & 0 & 0 & 1 \\
 \end{array}
 \right);  \ t\in\KK\setminus\{0\}, \ a,b,c \in \KK \right\}.
$$
In this case we can not find a unipotent subgroup of $G$ of complementary dimension that has trivial intersection with $G^{\mult}$.
Notice also that the center of $G^{\mult}$ is one-dimensional and unipotent. 
\end{example}

\section{Proofs}
\label{sec3}

\begin{proof}[Proof of Proposition~\ref{propim}]
Any element of a $\Gm$-subgroup in $G$ is semisimple, and any semisimple element is contained in a subtorus ~\cite[Theorem~22.2]{Hum}, so it is a product of elements from $\Gm$-subgroups. This proves the equivalence of $(1)$ and $(2)$. 

Any $\Gm$-subgroup is reductive, so it is contained in a maximal reductive sugbroup, i.e., in a subgroup conjugated to $G^{\red}$. This gives the implication $(1)\Rightarrow (4)$. The implication $(4)\Rightarrow (2)$ follows from the following lemma. 

\begin{lemma} \label{lemred}
Let $F$ be a connected reductive group. Then any element of $F$ is a product of at most two semisimple elements. 
\end{lemma}

\begin{proof}
The set of all semisimple elements in $F$ contains a nonempty subset $V$ that is open in $F$~\cite[Theorem~22.2]{Hum}. By~\cite[Lemma~7.4]{Hum}, the subset $V\cdot V$ coincides with $F$. 
\end{proof}

Since elements of finite order are dense in $\Gm$ and any non-unit element of $\Ga$ has infinite order, any homomorphism from $\Gm$ to $\Ga$ is trivial. So we obtain $(1)\Rightarrow (3)$. 

To prove $(3)\Rightarrow (1)$, assume that the normal subgroup $G^{\mult}$ is a proper subgroup of~$G$. Then the factorgroup $W:=G/G^{\mult}$ is a linear algebraic group and any element in $W$ is the image of a unipotent element in $G$ under the projection $G\to W$. So all elements in $W$ are unipotent, and $W$ is a unipotent linear algebraic group. Since $W$ is solvable, it admits a surjective homomorphism to a commutative unipotent group of positive dimension. A~commutative unipotent group is the direct product of $\Ga$-subgroups. This proves that there is a surjective homomorphism $W\to\Ga$. The composition $G\to W\to\Ga$ is the desired surjection. 

Let us prove that $(3)$ is equivalent to $(5)$. Since $[G^{\red},G^{\red}]=G^{\ses}$, we conclude that $[G,G]=G^{\ses}\ltimes L$, where $L\subseteq\Ru(G)$ is a normal subgroup. Since the group $\Ga$ is commutative, any homomorphism $G\to\Ga$ can be decomposed as $G\to G/[G,G]\to \Ga$. Since $G/[G,G]$ is a connected commutative 
linear algebraic group, it is isomorphic to the direct product of a torus and a commutative unipotent group~\cite[Theorem~3.2.8]{OV}. So $(3)$ is equivalent to the condition that $G/[G,G]$ contains no non-unit unipotent element. This is equivalent to $L=\Ru(G)$, and so to $(5)$. 

\smallskip

This completes the proof of Proposition~\ref{propim}.
\end{proof}

\begin{proof}[Proof of Theorem~\ref{tmain}.]
By Lemma~\ref{lemred}, the subgroup $G^{\red}$ is contained in $G^{\mult}$, and the decomposition $G=G^{\red}\ltimes \Ru(G)$ implies 
$G^{\mult}=G^{\red}\ltimes U$ for some subgroup $U$ in~$\Ru(G)$. 

\smallskip
Let $T$ be a maximal torus in $G^{\red}$ and $\uu=\Lie(U)$. 

\begin{lemma} \label{l2}
If $v$ is a $T$-semiinvariant vector in $\nn$ of nonzero weight, then $v\in\uu$.
\end{lemma}

\begin{proof}
Let $C$ be the one-parameter subgroup in $\Ru(G)$ with the tangent vector $v$. Fix an isomorphism $\Ga\to C, \, a\mapsto c(a)$. We have 
$$
tc(a)t^{-1}=c(\chi(t)a)
$$ 
for all $t\in T$, $a\in\KK$ and some nonzero weight $\chi$ of the torus $T$. 
By Proposition~\ref{propim}, the subgroup $G^{\mult}$ contains all subgroups conjugated to $G^{\red}$. So the element
$$
c(\chi(t)a-a)=(tc(a)t^{-1})c(-a)=t(c(a)t^{-1}c(a)^{-1})
$$
is contained in $G^{\red}$ for all $t\in T$ and $a\in\KK$. This proves that the subgroup $C$ is contained in $U$ and its tangent vector $v$ is in $\uu$. 
\end{proof}

\begin{lemma} \label{l3}
The subspace $\nn_1$ is contained in $\uu$.
\end{lemma}

\begin{proof}
Any simple $G^{\red}$-submodule $V$ in $\nn$ contains a highest weight vector $v$ with respect to some Borel subgroup $B$ in $G^{\red}$ containing the maximal torus $T$. If the module $V$ is non-trivial, its highest weight is nonzero. By Lemma~\ref{l2}, the vector $v$ is in $\uu$. Since $\uu$ is $G^{\red}$-invariant, we have $V\subseteq\uu$. This proves that $\nn_1\subseteq\uu$. 
\end{proof}

Lemma~\ref{l3} implies $\sss\subseteq\uu$ and so $G^{\red}\ltimes \exp(\sss)\subseteq G^{\mult}$.  

\begin{lemma} \label{l4}
Let $\rr=\Lie(G^{\red})$. Then the subalgebra $\rr\oplus\sss$ is an ideal in $\gg$. 
\end{lemma}

\begin{proof}
It suffices to check that for any simple $G^{\red}$-submodule $V$ in $\nn$ the subspace $[V,\rr\oplus\sss]$ is contained in $\sss$. 
We have 
$$
[V,\rr\oplus\sss]=[V,\rr]+[V,\sss]\subseteq \nn_1+[V,\sss], 
$$
so it remains to prove that $[V,\sss]\subseteq\sss$. 

\smallskip
If $V$ is non-trivial, then $V\subseteq\nn_1$, and the conclusion follows from the definition of $\sss$. Assume that $V$ is one-dimensional trivial. 
Since the Lie algebra $\sss$ is generated by $\nn_1$, it suffices to show that $[V,\nn_1]\subseteq\nn_1$. 

\smallskip
Take a non-trivial simple $G^{\red}$-submodule $W$ in $\nn_1$. It suffices to prove that $[V,W]\subseteq\nn_1$. The map
$$
V\otimes W\to\nn, \quad \sum_i v_i\otimes w_i\mapsto \sum_i [v_i,w_i]
$$
is $G^{\red}$-equivariant. Since $V$ is one-dimensional trivial, the module $V\otimes W$ is isomorphic to~$W$. Take a nonzero vector $v\in V$. 
The map $W\to [V,W]$, $w\mapsto [v,w]$ is surjective and $G^{\red}$-equivariant. So either it is an isomorphism, and then $[V,W]\subseteq\nn_1$ by the definition of~$\nn_1$, or
we have $[V,W]=0$. This concludes the proof of Lemma~\ref{l4}.
\end{proof}

\begin{lemma} \label{l5}
The ideal $\rr\oplus\sss$ is the minimal ideal in $\gg$ containing $\rr$.
\end{lemma} 

\begin{proof}
If some ideal of $\gg$ contains $\rr$, then by Lemmas~\ref{l2} and~\ref{l3} it contains the subalgebra $\sss$. On the other hand, Lemma~\ref{l4} claims that $\rr\oplus\sss$ is an ideal.
\end{proof} 

Since any $\Gm$-subgroup is contained in a subgroup conjugated to $G^{\red}$, we conclude from Lemma~\ref{lemred} that $G^{\mult}$ is the subgroup generated by all subgroups conjugated to $G^{\red}$. In other words, $G^{\mult}$ is the minimal normal subgroup of $G$ containing $G^{\red}$. By~\cite[Theorem~13.3]{Hum} this is equivalent to the fact that
$\Lie(G^{\mult})$ is the minimal ideal of $\gg$ containing $\rr$. Lemma~\ref{l5} implies $\Lie(G^{\mult})=\rr\oplus\sss$ and so 
$G^{\mult}=G^{\red}\ltimes\exp(\sss)$. 

\smallskip

This completes the proof of Theorem~\ref{tmain}.
\end{proof}

\begin{proof}[Proof of Corollary~\ref{cor1}.]
As $G$ is a regular subgroup, the nilpotent ideal $\nn=\Lie(\Ru(G))$ is a linear span of some root vectors in the reductive Lie algebra $\Lie(F)$. Since all roots are nonzero $T$-weights, Lemma~\ref{l2} implies $\nn=\nn_1=\sss$, so $G=G^{\mult}$. 
\end{proof}

\section{Concluding remarks and problems}
\label{sec4}

It is shown in~\cite{Chi-1, Chi-2} that for any nonzero nilpotent element $x$ in the Lie algebra $\gg=\sl_n(\KK)$ or $\sp_{2n}(\KK)$ there is a nilpotent $y$ such that $x$ and $y$ 
generate~$\gg$. This result implies that for any $\Ga$-subgroup $H_1$ in the group $G=\SL_n(\KK)$ or $\Sp_{2n}(\KK)$ there is a $\Ga$-subgroup $H_2$ such that
$H_1$ and $H_2$ generate the group $G$. In particular, the groups $\SL_n(\KK)$ and $\Sp_{2n}(\KK)$ are generated by two $\Ga$-subgroups. 

At the same time, any $\Ga$-subgroups $H_1,\ldots,H_d$ in the commutative unipotent group $\Ga^n$ define a homomorphism 
$$
H_1\times\ldots\times H_d \to \Ga^n, \quad (h_1,\ldots,h_d) \mapsto h_1\ldots h_d,
$$
whose image is of dimension at most $d$. So the group $\Ga^n$ is generated by at least $n$ one-parameter subgroups. 

\begin{problem}
Given an additively generated linear algebraic group $G$, what is the minimal number $k$ such that some $\Ga$-subgroups $H_1,\ldots,H_k$ generate $G$? 
\end{problem}

\begin{problem}
Given a multiplicatively generated linear algebraic group $G$, what is the minimal number $s$ such that some $\Gm$-subgroups $M_1,\ldots,M_s$ generate $G$? 
\end{problem}

Lemma~\ref{lemred} motivates the following two problems. 

\begin{problem} \label{qq}
Given an additively generated linear algebraic group $G$, what is the minimal number $l$ such that any element of $G$ is a product of at most $l$ unipotent elements? 
\end{problem}

\begin{problem}
Given a multiplicatively generated linear algebraic group $G$, what is the minimal number $m$ such that any element of $G$ is a product of at most $m$ semisimple elements? 
\end{problem}

When this note appeared as a preprint, Alexey Galt sent us the following solution to Problem~\ref{qq}. We assume that $G$ is not unipotent, so the subgroup $G^{\ses}$ is non-trivial.

\begin{proposition}
If the subgroup $G^{\ses}$ is adjoint, then $l=2$; otherwise $l=3$. 
\end{proposition}

\begin{proof}
Note that a product of two commuting unipotent elements is unipotent. Also if $(u,w)\in G^{\ses}\ltimes\Ru(G)$ and $u$ is unipotent, then $(u,w)$ is unipotent as well. 

It is known that in a semisimple group any non-central element is a product of two unipotent elements; see~\cite[Corollary of Theorem~3]{EG} for this result
in the context of Chevalley groups. So if $G^{\ses}$ is adjoint then any $g\in G^{\ses}$ equals $u_1u_2$ for some unipotent elements $u_1,u_2\in G^{\ses}$, and any $(g,u)\in G^{\ses}\ltimes\Ru(G)$ equals $(u_1,e)(u_2,u)$. We conclude that $l=2$. 

If $G^{\ses}$ is not adjoint, $z\in G^{\ses}$ is a non-unit central element, and $z=u_1u_2$, then $u_1u_2=z=u_1^{-1}zu_1=u_2u_1$. This shows that $z$ is unipotent, a contradiction. So $\l\ge 3$. At the same time, for a non-unit unipotent $v\in G^{\ses}$ the element $zv$ is not central, so $zv=u_1u_2$ and
$z=u_1u_2v^{-1}$. This proves that any $(z,u)\in G^{\ses}\ltimes\Ru(G)$ equals $(u_1,e)(u_2,e)(v^{-1},u)$, and so $l=3$. 
\end{proof} 

The final theme is about surjective morphisms from affine spaces. We say that an algebraic variety $X$ is an \emph{A-image} if there is a surjective morphism $\varphi\colon\AA^d\to X$
for some positive integer $d$. In this case we have $\KK[X]^{\times}=\KK^{\times}$ and the variety $X$ is irreducible and unirational. The last condition means that the field of rational functions $\KK(X)$ can be embedded into the field of fractions $\KK(x_1,\ldots,x_d)$ of the polynomial algebra $\KK[x_1,\ldots,x_d]$. 

One may ask whether any unirational variety $X$ without non-constant invertible functions is an $A$-image. This holds when $X$ is complete~\cite[Theorem~1.7]{AKaZ} and when $X$ is an affine cone~\cite[Theorem~1]{Ar-1}. 

Let $G$ be an additively generated linear algebraic group. By~\cite[Proposition~1.5]{AFKKZ}, there is a sequence $H_1,\ldots, H_d$ of $\Ga$-subgroups in $G$ such that the morphism
$$
H_1\times\ldots\times H_d\to G, \quad (h_1,\ldots,h_d)\mapsto h_1\ldots h_d
$$
is surjective. This shows that any additively generated linear algebraic group is an A-image. Moreover, this allows to prove that a homogeneous space $G/H$ of a connected linear
algebraic group $G$ is an A-image if and only if $\KK[G/H]^{\times}=\KK^{\times}$; see~\cite[Theorem~C]{Ar-2}.

\smallskip 

Denote by $\AA^1_*$ the punctured affine line $\AA^1\setminus\{0\}$. Let us say that an algebraic variety $X$ is an \emph{M-image} if there is a surjective morphism $\psi\colon(\AA^1_*)^d\to X$ for some positive integer $d$. In this case the variety $X$ is again irreducible and unirational, but we have no restriction on invertible functions. 

Since there is a surjective morphism
$$
\AA^1_*\to\AA^1, \quad x\mapsto x+\frac{1}{x},
$$
any $A$-image is an $M$-image. One may ask whether any unirational algebraic variety is an $M$-image. 

Let us return to algebraic groups. Since any connected linear algebraic group $G$ is generated by $\Ga$- and $\Gm$-subgroups, again by~\cite[Proposition~1.5]{AFKKZ} there is a sequence
$S_1,\ldots, S_d$ of connected one-parameter subgroups such that the morphism
$$
S_1\times\ldots\times S_d\to G, \quad (s_1,\ldots,s_d)\mapsto s_1\ldots s_d
$$
is surjective. This shows that any homogeneous space $G/H$ is an $M$-image. 

\bigskip

\emph{Acknowledgements.} \ The idea of this note appeared during the conference Vavilov Memorial 2024 in Saint Petersburg. The author is grateful to Eugene Plotkin for stimulating questions and discussions. 

Specials thanks are due to the reviewers for many useful comments and suggestions. 


\end{document}